\documentclass[12pt,a4paper]{amsart} 
\usepackage[a4paper,left=2.5cm,right=2.5cm,top=3cm,bottom=3cm]{geometry}
\usepackage{amsmath, amsthm, amsfonts, amssymb, graphicx, hyperref, bm,verbatim,mathrsfs,siunitx,natbib,tikz-cd}
\usepackage{stmaryrd,enumerate}
\theoremstyle{plain}
\newtheorem{thm}{Theorem}[section]
\newtheorem{lem}[thm]{Lemma}
\newtheorem{prop}[thm]{Proposition}
\newtheorem{cor}[thm]{Corollary}

\newtheorem{defn}[thm]{Definition}
\newtheorem*{rem}{Remark}

\newtheorem{conj}[thm]{Conjecture}

\numberwithin{sublem}{thm} 
\numberwithin{equation}{section}
\renewcommand{\mod}[1]{\  (\text{mod }{#1})}
\renewcommand*{\d}{\, \textup{d}}
\renewcommand{\gcd}{\operatorname{gcd}}

\newcommand{\e}{\varepsilon}
\newcommand{\Mod}[1]{\ (\mathrm{mod}\ #1)}
\renewcommand{\mod}[1]{\ \mathrm{mod}\ #1}
\numberwithin{equation}{section}
\makeatletter
\@namedef{subjclassname@2020}{%
  \textup{2020} Mathematics Subject Classification}
\makeatother
\usepackage[T1]{fontenc}
\numberwithin{equation}{section}
\subjclass[2020]{Primary 11N05, 11N13; Secondary 11N36}
\keywords{prime numbers, sieve methods, residue classes}

\begin{document}
\title{Residue Class Patterns of Consecutive Primes}
\author{Cheuk Fung (Joshua) Lau}
\address{{Mathematical Institute}, {University of Oxford}, {{Radcliffe Observatory Quarter, Woodstock Road}, {OX2 6GG}, {Oxford}, {United Kingdom}}}
\email{joshua.lau@maths.ox.ac.uk}
\begin{abstract}
Dickson's conjecture and the Hardy--Littlewood prime tuple conjecture predict that every pattern of reduced residue classes modulo $q$ is attained by infinitely many strings of $m$ consecutive primes. At present, however, even proving that a single non-constant residue class pattern of length $m$ occurs infinitely often is beyond the reach of existing methods. Combining Dirichlet's theorem on primes in arithmetic progressions with a theorem of Shiu (2000) shows that, for any $m,q\in\mathbb N$ with $q \ge 3$, at least $m\varphi(q)$ residue class patterns of length $m$ are attained by infinitely many consecutive primes.\\

In this paper, we prove that if $q$ is squarefree, every prescribed sequence of at least $60m\log m$ reduced residue classes mod $q$ contains, in order, an $m$-term block pattern that occurs infinitely often among consecutive primes, with each constant block of length at most $\lceil\log m\rceil$. A recursive combinatorial argument then shows that if $q$ is squarefree and $q \gg (\log m)^2$, then at least
\[
\gg \frac{m}{(\log m)^{10}} \varphi(q)^2
\]
residue class patterns of length $m$ occur infinitely often among consecutive primes. Moreover, we also show that if $q$ is squarefree and $q \gg (\log m)^2$, then at least
\[
\gg e^{-O(m \log_2 m/\log m)} \varphi(q)^{m/\lceil \log m \rceil}
\]
residue class patterns of length $m$ occur infinitely often among consecutive primes. The proof consists of a modification of the Maynard--Tao sieve found in Banks, Freiberg, and Maynard (2016), by considering the $r$-th moment instead of the 2nd moment for an integer $r$ depending on $m$, which is then combined with an Erd\H{o}s--Rankin type construction.
\end{abstract}
\maketitle
\tableofcontents
\section{Introduction}
A central theme in analytic number theory is that, after accounting for the obvious local congruence obstructions, the prime numbers should behave like a random set. This philosophy underlies the Hardy--Littlewood prime tuple conjecture and predicts that every admissible local pattern of primes should occur infinitely often. One way to study this phenomenon is to examine the residue classes of consecutive primes modulo a fixed integer $q$.\\

For $q,m \in \mathbb{N}$ with $q \ge 2$, $x \in \mathbb{R}$ and $\mathbf{a} \in \prod_{i=1}^m (\mathbb{Z}/q\mathbb{Z})^\times$, define $$\pi(x;q,\mathbf{a})=\#\{p_n \leq x:p_{n+i-1} \equiv a_i \pmod{q} \text{ for all }i=1,2,\ldots,m\},$$ where $p_i$ denotes the $i$-th prime. Thus $\pi(x;q,\mathbf{a})$ counts occurrences of a prescribed residue class pattern among $m$ consecutive primes.\\

Understanding these patterns is closely related to the broader question of how randomly distributed consecutive primes are. Numerical investigations and probabilistic models suggest that, apart from congruence restrictions, every residue class pattern should occur infinitely often. A consequence of Dickson's conjecture states the following.
\begin{conj} \label{conj:dick}
    For any $q,m \in \mathbb{N}$ with $q \ge 2$ and $\mathbf{a} \in \prod_{i=1}^m (\mathbb{Z}/q\mathbb{Z})^\times$, $\pi(x;q,\mathbf{a}) \to \infty$ as $x \to \infty$. Equivalently,
    $$\#\left\{\mathbf{a} \in \prod_{i=1}^m (\mathbb{Z}/q\mathbb{Z})^\times: \pi(x;q,\mathbf{a}) \to \infty \text{ as } x 
    \to \infty\right\}=\varphi(q)^m.$$
\end{conj}
Additionally, the Hardy--Littlewood prime tuple conjecture predicts the quantitative asymptotic behaviour of $\pi(x;q,\mathbf{a})$. Recently, more refined heuristics were developed by \citet{lemke2016unexpected}. Their work predicts substantial biases for finite ranges, for example, consecutive primes tend to avoid repeating the same residue class modulo $q$ more often than naive randomness would suggest, but these biases are expected to disappear asymptotically.\\

Despite these strong heuristic predictions, our current understanding remains extremely limited. Dirichlet's theorem on primes in arithmetic progressions states that $\pi(x;q,a) \to \infty$ as $x \to \infty$ for any $a \in (\mathbb{Z}/q\mathbb{Z})^\times$, so Conjecture \ref{conj:dick} is known when $m=1$. For general $m \in \mathbb{N}$, we have the following result of \citet{shiu2000strings}.
\begin{thm}[Shiu] \label{thm:shiu}
For any $m,q \in \mathbb{N}$ with $q \ge 2$ and $a \in (\mathbb{Z}/q\mathbb{Z})^\times$, let $\mathbf{a}=(a,\ldots,a) \in \prod_{i=1}^m (\mathbb{Z}/q\mathbb{Z})^\times$. Then, $\pi(x;q,\mathbf{a}) \to \infty$ as $x \to \infty$.
\end{thm}
Moreover, \citet{maynard2016dense} proved that $\pi(x;q,(a,\ldots,a)) \gg \pi(x)$. However, even proving that a single non-constant residue class pattern of length $m$ occurs infinitely often among consecutive primes is beyond the reach of existing methods. In an easier setting of consecutive sums of two squares instead of primes, \citet{kimmel2024consecutive,kimmel2025positive} considered
$$\mathbf{E} = \{c^2+d^2:c,d \in \mathbb{N}\}=\{E_n:n \in \mathbb{N}\},$$
and call $a \in \mathbb{Z}/q\mathbb{Z}$ $\mathbf{E}$-admissible if there is $c,d \in \mathbb{N}$ such that $c^2+d^2 \equiv a \Mod{q}$. For an $m$-tuple of admissible residue classes $\mathbf{a}=(a_1,...,a_m) \in (\mathbb{Z}/q\mathbb{Z})^m$, let
\begin{align*}
    N(x,q,\mathbf{a})=\#\{E_n \le x:E_{n+i-1} \equiv a_i \Mod{q} \ \forall 1 \le i \le m\}.
\end{align*}
\citet{kimmel2024consecutive} proved that for any $q \in \mathbb{N}$ with $q \ge 2$ and $a_1,a_2,a_3 \in \mathbb{Z}/q\mathbb{Z}$ $\mathbf{E}$-admissible, $N(x,q,(a_1,a_2,a_3)) \to \infty$ as $x \to \infty$. For a general $m \in \mathbb{N}$ and squarefree $q \in \mathbb{N}$, \citet{kimmel2025positive} proved that for any $m$-tuple $\mathbf{a}$ that is a concatenation of two constant tuples $(a_1,\ldots,a_1)$ and $(a_2,\ldots,a_2)$ with $a_1,a_2 \in (\mathbb{Z}/q\mathbb{Z})^\times$, we have $N(x,q,\mathbf{a}) \to \infty$ as $x \to \infty$. Even in this easier setting, we are very far from proving that all $\mathbf{E}$-admissible tuples of residue classes are indeed attained by infinitely many consecutive sums of two squares.\\

In this paper, we investigate the number of residue class patterns that are attained by infinitely many consecutive primes. One may use Dirichlet's theorem to show that for any $q,m \in \mathbb{N}$ with $q \ge 2$, there are at least $\varphi(q)$ many $m$-tuples $\mathbf{a} \in \prod_{i=1}^m (\mathbb{Z}/q\mathbb{Z})^\times$ such that
$\pi(x;q,\mathbf{a}) \to \infty \text{ as }x \to \infty.$ Using \citet{shiu2000strings} along with Dirichlet's theorem, we may obtain a better lower bound.
\begin{prop} \label{shiudirichlet}
    For any $q,m \in \mathbb{N}$ with $q \ge 3$, there are at least $m\varphi(q)$ many $m$-tuples $\mathbf{a} \in \prod_{i=1}^m (\mathbb{Z}/q\mathbb{Z})^\times$ such that
    $\pi(x;q,\mathbf{a}) \to \infty \text{ as }x \to \infty.$
\end{prop}
\begin{proof}
    For each $a \in (\mathbb{Z}/q\mathbb{Z})^\times$, Shiu's theorem states $\pi(x;q,(a,\ldots,a)) \to \infty$ as $x \to \infty$. By Dirichlet's theorem, we know that each string of consecutive primes all congruent to $a \pmod{q}$ must terminate. Since $\varphi(q) \ge 2$, by the pigeonhole principle there must exist $a' \in (\mathbb{Z}/q\mathbb{Z})^\times$ with $a' \neq a$ such that $\pi(x;q,(a,\ldots,a,a')) \to \infty$ as $x \to \infty$. Repeating this `shifting' argument $m-1$ more times, we obtain for each $a \in (\mathbb{Z}/q\mathbb{Z})^\times$ there are $m$ many $m$-tuples attained by infinitely many consecutive primes. By considering the first entry, the $m$-tuples obtained in this way for distinct $a \in (\mathbb{Z}/q\mathbb{Z})^\times$ are distinct, so in total we can obtain at least $m \varphi(q)$ such tuples.
\end{proof}
In fact, this lower bound is optimal if the only property of primes we use is Theorem \ref{thm:shiu}.
Let $q \in \mathbb{Z}_{>1}$, and let $(\mathbb{Z}/q\mathbb{Z})^\times=\{c_1,\ldots,c_{\varphi(q)}\}$. Consider the sequence
\begin{align*}
    b_n &= c_k, \text{ where } k-1 < n \le k, \ 1 \le k \le \varphi(q),\\
    b_{\varphi(q)+n} &= c_k, \text{ where } 2(k-1) < n \le 2k, \ 1 \le k \le \varphi(q),\\
    b_{3\varphi(q)+n} &= c_k, \text{ where } 3(k-1) < n \le 3k, \ 1 \le k \le \varphi(q),
\end{align*}
and so on, i.e.
\begin{align*}
    b_{r(r-1)\varphi(q)/2+n} &= c_k, \text{ where } (k-1)r < n \le kr, \ 1 \le k \le \varphi(q), \ r \ge 1.
\end{align*}
For $m \in \mathbb{N}$ and $\mathbf{a}=(a_1,\ldots,a_{m}) \in \prod_{i=1}^{m} (\mathbb{Z}/q\mathbb{Z})^\times$, define the corresponding counting function
\begin{align*}
    \widetilde \pi_{q} (x;\mathbf{a})=\#\{b_n \leq x:b_{n+i-1} \equiv a_i \Mod{q} \text{ for all }i=1,2,\ldots,m\}.
\end{align*}
We observe that for any $m \in \mathbb{N}$, $a \in (\mathbb{Z}/q\mathbb{Z})^\times$, and $\mathbf{a}=(a,\ldots,a) \in \prod_{i=1}^m (\mathbb{Z}/q\mathbb{Z})^\times$, the sequence $(b_n)_{n=1}^\infty$ satisfies $\widetilde \pi_{q}(x;\mathbf{a}) \to \infty$ as $x \to \infty$. However, there are exactly $m\varphi(q)$ many $m$-tuples $\mathbf{a} \in \prod_{i=1}^m (\mathbb{Z}/q\mathbb{Z})^\times$ such that $\widetilde{\pi}_{q}(x;\mathbf{a}) \to \infty$ as $x \to \infty$, namely
\begin{align*}
    (\underbrace{c_t,\ldots,c_t}_{r \text{ times}},\underbrace{c_{t+1},\ldots,c_{t+1}}_{m-r\text{ times}}) \in \prod_{i=1}^m(\mathbb{Z}/q\mathbb{Z})^\times, \quad 1 \le r \le m, \quad 1 \le t \le \varphi(q), \quad c_{\varphi(q)+1} := c_1.
\end{align*}

For general $m \in \mathbb{N}$, despite knowing that $m \varphi(q)$ tuples of residue classes are attained by infinitely many consecutive primes, it is currently not known whether any other specific tuple of residue classes is attained by infinitely many consecutive primes. To state our main results, we make the following definitions.
\begin{defn} \label{defn:notationforthms}
    For a squarefree{} integer $q \ge 2$, $m \in \mathbb{Z}^+$, and $A \subseteq \prod_{i=1}^m (\mathbb{Z}/q\mathbb{Z})^\times$, define
    \[
    \pi(x;q,A)=\#\{p_n \leq x:(p_n \mod{q},\ldots,p_{n+m-1} \mod{q}) \in A\}.
    \]
    Also, for $r \in \mathbb{Z}^+$ with $r>1$, define 
    \[
    M{}=\left\lceil \left(\frac{2^{3r-2} (r-1)^{2r-1}}{r!}\right)^{\frac{1}{r-1}} m(m(r-1)+r)^{\frac{1}{r-1}}  \right\rceil,
    \]
    and define the set of functions
    \small
    $$J_r(m,M)=\{j:\{1,\ldots,m\} \to \{1,\ldots,M\}:j(i+1) \geq j(i), \text{ no consecutive $r$ values }j(i) \text{ are equal}\}.$$
    \normalsize
\end{defn}
The main result of this paper is the following theorem, proven in Section \ref{mainresultsection}.
\begin{thm} \label{mainmodthm}
    Let $q$ be a squarefree{} integer, and $r,m \in \mathbb{Z}^+$ with $r>1$. Recall $\pi(x;q,A)$, $M$, and $J_r(m,M)$ from Definition \ref{defn:notationforthms}.
    For any $a_1,\ldots,a_{M{}} \in (\mathbb{Z}/q\mathbb{Z})^\times$, let
    $$A=\{(c_1,\ldots,c_m): \exists j \in J_r(m,M) \text{ s.t. }c_i=a_{j(i)} \forall 1 \leq i \leq m\}.$$
    Then $\pi(x;q,A) \to \infty$ as $x \to \infty$.
\end{thm}
Using this, one can argue combinatorially to obtain the following result for $q$ in `medium' range, which is proven in Section \ref{maincorsection}.
\begin{cor} \label{smallcorintro}
     For any $0<c<1$, if $q \ge 2$ is squarefree and $\varphi(q) > 8c^{-1}e^2 (\log m)^2$, then for $m$ sufficiently large,$$\#\left\{\mathbf{a} \in \prod_{i=1}^m (\mathbb{Z}/q\mathbb{Z})^\times:\lim_{x \to \infty} \pi(x;q,\mathbf{a})=\infty \right\} \geq \frac{\lfloor (1-c)m \rfloor c^5}{256e^{10}(\log m)^{10}}\varphi(q)(\varphi(q)-1).$$
\end{cor}
Choosing $c=5/6$, this gives a better lower bound than that of Proposition \ref{shiudirichlet} when $$\varphi(q)>3823e^{10} (\log m)^{10}+1.$$ In Section \ref{maincorsection}, we also obtain a corresponding lower bound when $q$ is in a `large' range.
\begin{cor} \label{largecorintro}
    For $m,r \in \mathbb{Z}^+$, recall $M$ from Definition \ref{defn:notationforthms}. If $q \ge 2$ is squarefree and $\varphi(q) \geq M{}$, there are at least
    $$\frac{\lceil m/(r-1) \rceil!}{M{}(M{}-1)\cdots(M{}-\lceil m/(r-1) \rceil+1)} \cdot \varphi(q)(\varphi(q)-1) \cdots (\varphi(q)-\lceil m/(r-1) \rceil+1)$$
    $m$-tuples $\mathbf{a}$ such that $\pi(x;q,\mathbf{a}) \to \infty$ as $x \to \infty$.
\end{cor}
For example, for $m=2$, we may choose $r=2$ in Corollary \ref{largecorintro} to show that for any $q$ squarefree and $\varphi(q)\ge 64$,
\begin{align*}
    \#\{(a_1,a_2) \in (\mathbb{Z}/q\mathbb{Z})^\times \times (\mathbb{Z}/q\mathbb{Z})^\times:\lim_{x \to \infty} \pi(x;q,(a_1,a_2))=\infty\} \ge \frac{1}{2016}\varphi(q)(\varphi(q)-1).
\end{align*}
In particular, when $\varphi(q)>4033$, we beat the lower bound in Proposition \ref{shiudirichlet}. This misses the conjectured lower bound of $\varphi(q)^{2}$ by a constant. Putting $r=\lceil \log m +1 \rceil$ in Corollary \ref{largecorintro}, we get
\begin{cor} \label{cor:largecorintrosimplified}
    If $q \ge 2$ is squarefree{} and $\varphi(q) > 8e^2 m\log m$, then for $m$ sufficiently large,
    $$\# \left\{ \mathbf{a} \in \prod_{i=1}^m (\mathbb{Z}/q\mathbb{Z})^\times:\lim_{x \to \infty} \pi(x;q,\mathbf{a})=\infty \right\} \gg e^{-O(m \log_2 m/\log m)} \varphi(q)^{m/\lceil \log m \rceil}.$$
\end{cor}
In both corollaries, we chose $r=\lceil \log m+1 \rceil$. Heuristically this is because for $m$ large, we have $M=\Theta(r^{1+o(1)} m^{1+1/r + o(1/r)})$. To minimise $M$, we choose $r \approx \log m$. 
\section{Outline}
The proof of Theorem \ref{mainmodthm} has two main ingredients. First, in Section \ref{sec:modifiedmaynardtaosieve} we develop a modification of the Maynard--Tao sieve which guarantees the existence of several clusters of primes satisfying prescribed spacing conditions. Second, in Section \ref{sec:erdosrankin} we adapt the Erdős--Rankin construction so that these clusters consist of consecutive primes lying in prescribed residue classes modulo $q$. Combining these two ingredients yields Theorem \ref{mainmodthm}, from which the lower bounds for the number of attainable residue class patterns follow by a combinatorial argument in Section \ref{maincorsection}.\\

A finite set of integers $\mathcal{H}$ is said to be \emph{admissible} if for
every prime $p$,
\[
\#\left\{
n \Mod p :
\prod_{h\in\mathcal{H}} (n+h)\equiv 0 \Mod p
\right\}
<p.
\]
Fix $m,r\in\mathbb{Z}^+$ with $r>1$, and let $q$ be squarefree. Define the
parameter $M$ as in Definition \ref{defn:notationforthms}. For sufficiently large $N$, let
$W$ denote the usual product of small primes appearing in the Maynard--Tao
sieve. We also introduce the Maynard--Tao sieve weights
\[
w_n=
\left(
\sum_{\substack{d_1,\ldots,d_k\\ d_i\mid qn+h_i}}
\lambda_{d_1,\ldots,d_k}
\right)^2,
\]
where the coefficients $\lambda_{d_1,\ldots,d_k}$ are defined in
Section \ref{sec:modifiedmaynardtaosieve}. Our main sieve result (Proposition \ref{maynardtaosieve}) states that if
$\mathcal{H}=\{h_1,\ldots,h_k\}\subset [0,N]$ is an admissible $k$-tuple whose
pairwise differences are $\varepsilon\log N$-smooth, then one can choose a
residue class $b \Mod W$ satisfying suitable local congruence conditions so
that, for any partition
\[
\mathcal{H}=\mathcal{H}_1\cup\cdots\cup\mathcal{H}_M
\]
into equal-sized subsets, there exists $n\in[N,2N]$ for which $m+1$ of the sets
$qn+\mathcal{H}_i$ each contain between $1$ and $r-1$ primes, while every
intermediate block contains no primes.\\

The proof of Proposition \ref{maynardtaosieve} follows the ideas of \citet{banks2016limit} and \citet{merikoski2020limit}. Observe that it suffices to establish the positivity of the sum
\begin{align}
    S=\sum_{N<n \leq 2N} \Bigg( &\sum_{i=1}^k \bm{1}_{\mathbb{P}}(qn+h_i)-m(r-1) \notag{}\\
    &-(m(r-1)+r) \sum_{j=1}^M{} \sum_{\substack{h_{j_1},\ldots,h_{j_r} \in \mathcal{H}_j\\ j_1<\cdots<j_r}} \prod_{i=1}^r \bm{1}_{\mathbb{P}} (qn+h_{j_i}) \Bigg) w_n. \label{eqn:outlinekeysum}
\end{align}
To do this, in Lemma \ref{lem:maynardtaosieve} we first establish estimates for
the total weight, the weighted count of primes, and the weighted $r$-fold
correlations
\begin{align*}
    &\sum_{\substack{N<n\le 2N\\ n\equiv b\;(\bmod W)}} w_n,\\
&\sum_{\substack{N<n\le 2N\\ n\equiv b\;(\bmod W)}}
\mathbf{1}_{\mathbb{P}}(qn+h_j)\,w_n,\\
&\sum_{\substack{N<n\le 2N\\ n\equiv b\;(\bmod W)}}
\mathbf{1}_{\mathbb{P}}(qn+h_{j_1})
\cdots
\mathbf{1}_{\mathbb{P}}(qn+h_{j_r})\,w_n.
\end{align*}
These estimates are then combined to prove the positivity of \eqref{eqn:outlinekeysum}.\\

To apply Proposition \ref{maynardtaosieve}, Section \ref{sec:erdosrankin} develops a modified
Erd\H{o}s--Rankin construction. Given prescribed residue classes $r_1,...,r_k \Mod{q}$, in Lemma \ref{lemwithmod} we
construct an admissible tuple $\{h_1,\ldots,h_k\}$ with
$h_i\equiv r_i \Mod q$, whose pairwise differences are
$\varepsilon\log N$-smooth, together with the required residue class
$b\Mod W$. This construction also ensures that the primes detected by the
sieve are consecutive, completing the proof of Theorem \ref{mainmodthm}.\\

Finally, Section \ref{maincorsection} uses Theorem \ref{mainmodthm} in a recursive combinatorial
argument to obtain lower bounds for the number of residue class patterns
attained by infinitely many consecutive primes.
\section{Acknowledgements}
We would like to thank Jori Merikoski for suggesting this question, and for numerous helpful comments throughout the writing of this paper.
\section{Notation}
Throughout this paper, we use $\lfloor x \rfloor$ to denote the largest integer not greater than $x$, and $\lceil x \rceil$ to denote the least integer not less than $x$. We say $f \ll g$ and $f=O(g)$ when there exists a constant $C>0$ such that $|f(x)| \leq Cg(x)$ for $x$ sufficiently large. If the implied constant depends on parameter $\e$ say, then we write $f \ll_\e g $ or $f=O_\e(g)$. We use $f=o(g)$ to mean $\lim_{x \to \infty} f(x)/g(x)=0$.\\

Sums of the form $\sum_p$ range over primes, and $\mathbb{P}$ denotes the set of primes. We use $\bm{1}_{\mathbb{P}}(n)$ to denote the indicator function of whether $n \in \mathbb{P}$. Given integers $d_1,d_2$ we use $\gcd(d_1,d_2)$ or $(d_1,d_2)$ to denote the greatest common divisor of $d_1$ and $d_2$, and $\operatorname{lcm}(d_1,d_2)$ or $[d_1,d_2]$ to denote the least common multiple of $d_1$ and $d_2$. For a positive integer $q>1$, denote by $P^+(q)$ the largest prime factor of $q$. We use $\varphi(q)$ to denote the Euler totient function of $q$. Given integers $k$ and $n$, $\log_kn$ denotes the $k$-fold iterated logarithm of $n$ in base $e$, for example $\log_1 n=\log n$ and $\log_2n=\log \log n$.

\section{A Modified Maynard--Tao Sieve} \label{sec:modifiedmaynardtaosieve}
In order to use the methods of \citet{banks2016limit}, we need the following results.
\begin{lem} \label{exceptionalzero}
    Let $T \geq 3$ and $P \geq T^{1/\log_2 T}$. Among all primitive characters $\chi \Mod{q}$ with $q \leq T$ and $P^+(q) \leq P$, there exists at most one such character such that $L(s,\chi)$ has a zero in the region
    $$\Re(s)>1-\frac{c}{\log P}, \quad |\Im(s)| \leq \exp \left( \log P/\sqrt{\log T} \right), $$
    where $c$ is a positive absolute constant. If this character $\chi \Mod{q}$ exists and is real, then $L(s,\chi)$ has precisely one zero in the above region, which is simple and real, and satisfies
    $$P^+(q) \gg \log q \gg \log_2 T.$$
\end{lem}
\begin{proof}
    This is \cite[Lemma 4.1]{banks2016limit}.
\end{proof}
We fix the absolute constant $c$ in \citet{banks2016limit} and define $Z_T=P^+(q)$ if such exceptional modulus $q$ exists, and set $Z_T=1$ otherwise.
\begin{thm}[Modified Bomberi--Vinogradov] \label{modifiedbombierivinogradov}
    Let $N>2$. Fix any $C>0$,  $\theta=1/2-\delta \in (0,1/2)$ and $\e>0$. Suppose $q_0$ is a squarefree integer with $q_0<N^\e$ and $P^+(q_0)<N^{\e/\log_2 N}$. If $\e$ is sufficiently small in terms of $C,\delta,c$ in Lemma \ref{exceptionalzero}, then with $Z_{N^{2\e}}$ as above we have
    $$\sum_{\substack{q<N^\theta\\ q_0 \mid q\\ (q,Z_{N^{2\e}})=1}} \max_{(q,a)=1} \left| \psi(N;q,a)-\frac{\psi(N)}{\varphi(q)} \right| \ll_{\delta,C} \frac{N}{\varphi(q_0) (\log N)^C}.$$
\end{thm}
\begin{proof}
    This is \cite[Theorem 4.2]{banks2016limit}.
\end{proof}
Given a squarefree{} integer $q \ge 2$ and an admissible tuple $(h_1,\ldots,h_k)$, define the set
$$\mathcal{H}(n)=\{qn+h_1,\ldots,qn+h_k\}.$$
We define the sieve weights $\lambda_{d_1,\ldots,d_k}$ the same way as \cite{banks2016limit}, i.e. 
$$\lambda_{d_1,\ldots,d_k}=\left\{ \begin{aligned}
    \left( \prod_{i=1}^k \mu(d_i) \right) \sum_{j=1}^J \prod_{\ell=1}^k F_{\ell,j} \left( \frac{\log d_\ell}{\log N} \right), &\quad \text{ if } \gcd(d_1 \cdots d_k,Z_{N^{4 \e}})=1\\
    0, &\quad \text{ otherwise}
\end{aligned} \right.$$
for some fixed $J$, where $F_{\ell,j}:[0,\infty) \to \mathbb{R}$ are fixed smooth compactly supported functions that are not identically zero, with support condition
$$\sup \left\{ \sum_{\ell=1}^k t_\ell:\prod_{\ell=1}^k F_{\ell,j}(t_\ell) \neq 0 \right\} \leq \delta$$
for all $j=1,2,\ldots,J$ and some small $\delta>0$. Let
$$F(t_1,\ldots,t_k) := \sum_{j=1}^J \prod_{\ell=1}^k F_{\ell,j}'(t_\ell),$$
where $F_{\ell,j}'$ denotes the derivative of $F_{\ell,j}$. We also assume $F_{\ell,j}$ are chosen such that $F(t_1,\ldots,t_k)$ is symmetric. Since $J$ and $F_{\ell,j}$ for $1 \le \ell \le k$ and $1 \le j \le J$ are fixed, we have $\lambda_{d_1,\ldots,d_k} \ll 1$ uniformly in $d_1,\ldots,d_k$. We further define
$$w_n = \left( \sum_{\substack{d_1,\ldots,d_k\\d_i \mid qn+h_i}} \lambda_{d_1,\ldots,d_k} \right)^2,$$
and for $\e>0$ define $$W=\prod_{\substack{p \leq \e \log N\\ p \nmid Z_{N^{4 \e}}}} p, \quad B=\frac{\varphi(W)}{W} \log N.$$
We remark here, for $N$ sufficiently large in terms of $q$, we have $q \mid W$, so
$$\frac{\varphi(qW)}{qW}=\frac{\varphi(W)}{W}.$$
 We define the following quantities for $r \in \mathbb{Z}^+$ and $k \geq r$
\begin{align*}
    I_k(F) &:= \int_0^\infty \cdots \int_0^\infty F(t_1,\ldots,t_k) \d t_1 \cdots \d t_k,\\
    J_k^{(r)}(F) &:= \int_0^\infty \cdots \int_0^\infty \left( \int_0^\infty \cdots \int_0^\infty F(t_1,\ldots,t_k) \d t_{k-r+1} \cdots \d t_k \right)^2 \d t_1 \cdots \d t_{k-r}.
\end{align*}
\begin{lem} \label{mobiusinvlem}
    Let $N$ be sufficiently large in terms of $q$. If $F_1,\ldots,F_k,G_1,\ldots,G_k:[0,\infty) \to \mathbb{R}$ are compactly supported functions, then
    $$\sideset{}{'}\sum_{\substack{d_1,\ldots,d_k\\ d_1',\ldots,d_k'}} \prod_{j=1}^k \frac{\mu(d_j) \mu(d_j')}{[d_j,d_j']} F_j \left( \frac{\log d_j}{\log N} \right) G_j \left( \frac{\log d_j'}{\log N} \right)=(c+o(1)) B^{-k},$$
    where $\sum'$ denotes the additional restriction of $[d_1,d_1'],\ldots,[d_k,d_k'],qWZ_{N^{4 \e}}$ being pairwise coprime, and
    $$c=\prod_{j=1}^k \int_0^\infty F_j'(t_j) G_j'(t_j) \d t_j.$$
    The same result holds if $[d_j,d_j']$ are replaced by $\varphi([d_j,d_j'])$, i.e.
    \[
    \sideset{}{'}\sum_{\substack{d_1,\ldots,d_k\\ d_1',\ldots,d_k'}} \prod_{j=1}^k \frac{\mu(d_j) \mu(d_j')}{\varphi([d_j,d_j'])} F_j \left( \frac{\log d_j}{\log N} \right) G_j \left( \frac{\log d_j'}{\log N} \right)=(c+o(1)) B^{-k}.
    \]
\end{lem}
\begin{proof}
    If $N$ is sufficiently large in terms of $q$ such that $Z_{N^{4 \e}}>P^+(q)$ and $q \mid W$, then the additional restriction is the same as saying $[d_1,d_1'],\ldots,[d_k,d_k'],WZ_{N^{4 \e}}$ being pairwise coprime, which is just \cite[Lemma 4.5]{banks2016limit}.
\end{proof}
We have an estimate for $J_k^{(r)}(F)$ in terms of $I_k(F)$.
\begin{lem} \label{integralestimate}
    Let $0<\rho<1$ and $r \in \mathbb{Z}^+$ with $2 \leq r \leq k$. Then there is a fixed choice of $J$ and $F_{\ell,j}$ for $\ell \in \{1,2,\ldots,k\}$ and $j \in \{1,2,\ldots,J\}$ with the required properties such that
    \begin{align*}
        J_k^{(1)}(F) &\geq (1+O((\log k)^{-1/2})) \left( \frac{\rho \delta \log k}{k} \right) I_k(F),\\
        J_k^{(r)}(F) &\leq (1+O((\log k)^{-1/2})) \left( \frac{\rho \delta \log k}{k} \right)^r I_k(F).
    \end{align*}
\end{lem}
\begin{proof}
    The proof is similar to \cite[Lemma 4.7]{banks2016limit}. The result is trivial if $k$ is bounded, so assume $k$ is sufficiently large. Define $F_k=F_k(t_1,\ldots,t_k)$ by
    \begin{align*}
        F_k(t_1,\ldots,t_k)&=\left\{ \begin{aligned}
            \prod_{i=1}^k g(kt_i), &\quad \text{ if } \sum_{i=1}^k t_i \leq 1,\\
            0, &\quad \text{ otherwise},
        \end{aligned} \right.\\
        g(t) &= \left\{ \begin{aligned}
            \frac{1}{1+At}, &\quad \text{ if } t \in [0,T],\\
            0, &\quad \text{ otherwise},
        \end{aligned} \right.\\
        A &= \log k-2\log_2 k,\\
        T &= \frac{e^A-1}{A}.
    \end{align*}
    The first assertion follows from \cite[Lemma 4.7]{banks2016limit}. For the second assertion, we will first prove that for all $x \ge 0$ we have
    \begin{align}
        \left( \int_{t_1+\cdots+t_r \leq x} g(t_1) \cdots g(t_r) \d t_1 \cdots \d t_r \right)^2 &\leq (\log k)^r \int_{t_1 + \cdots+t_r \leq x} g(t_1)^2 \cdots g(t_r)^2 \d t_1 \cdots \d t_r. \label{eqn:intermediateintegralofg}
    \end{align}
    For $0 \leq x \leq \log k$, by Cauchy--Schwarz we have
    \begin{align*}
        \left( \int_{t_1+\cdots+t_r \leq x} g(t_1) \cdots g(t_r) \d t_1 \cdots \d t_r \right)^2 &\leq \frac{x^r}{2^r} \int_{t_1 + \cdots+t_r \leq x} g(t_1)^2 \cdots g(t_r)^2 \d t_1 \cdots \d t_r\\
        &\leq (\log k)^r \int_{t_1 + \cdots+t_r \leq x} g(t_1)^2 \cdots g(t_r)^2 \d t_1 \cdots \d t_r.
    \end{align*}
    We now prove \eqref{eqn:intermediateintegralofg} for $x\ge\log k$. Let $y=\min(x,T)$ and note $\log(1+Ay) \leq A$. Hence
    \begin{align*}
        \int_{t_1+\cdots+t_r \leq x} g(t_1)^2 \cdots g(t_r)^2 &\d t_1 \cdots \d t_r\\
        &\geq \int_{t_1+\cdots+t_r \leq y} g(t_1)^2 \cdots g(t_r)^2 \d t_1 \cdots \d t_r\\
        &= \int_{\substack{t_1+\cdots+t_{r-1} \leq y\\ t_1,\ldots,t_{r-1} \geq 0}} \frac{1}{(1+At_1)^2} \cdots \frac{1}{(1+At_{r-1})^2} \d t_1 \cdots \d t_{r-1}\\
        &\quad\times \int_{\substack{0 \leq t_r \leq y-(t_1+\cdots+t_{r-1})}} \frac{1}{(1+At_r)^2} \d t_r\\
        &\geq \int_{\substack{t_1+\cdots+t_{r-1} \leq y-1\\ t_1,\ldots,t_{r-1} \geq 0}} \frac{1}{(1+At_1)^2} \cdots \frac{1}{(1+At_{r-1})^2} \d t_1 \cdots \d t_{r-1}\\
        &\quad\times \int_{\substack{0 \leq t_r \leq 1}} \frac{1}{(1+At_r)^2} \d t_r\\
        &= \frac{1}{A+1}\int_{\substack{t_1+\cdots+t_{r-1} \leq y-1\\ t_1,\ldots,t_{r-1} \geq 0}} \frac{1}{(1+At_1)^2} \cdots \frac{1}{(1+At_{r-1})^2} \d t_1 \cdots \d t_{r-1}.
    \end{align*}
    Since $y \geq \log k \geq r$ for $k$ sufficiently large, we can do this $r-1$ more times, and we get
    \begin{align*}
        \int_{t_1+\cdots+t_r \leq x} g(t_1)^2 \cdots g(t_r)^2 \d t_1 \cdots \d t_r &\geq \frac{1}{(A+1)^r} \geq \frac{1}{(\log k)^r}
    \end{align*}
    for $k$ sufficiently large. Since the integral of $g$ over $[0,\infty)$ is 1, for all $x \geq 0$ we have
    \begin{align*}
        \left( \int_{t_1+\cdots+t_r \leq x} g(t_1) \cdots g(t_r) \d t_1 \cdots \d t_r \right)^2 &\leq
(\log k)^r \int_{t_1 + \cdots+t_r \leq x} g(t_1)^2 \cdots g(t_r)^2 \d t_1 \cdots \d t_r,
    \end{align*}
    so \eqref{eqn:intermediateintegralofg} holds for all $x \ge 0$. Therefore
    \begin{align*}
        J_k^{(r)}(F_k)&=\idotsint \limits_{\sum_{i=1}^{k-r} t_i \leq 1} \left( \prod_{i=1}^{k-r} g(kt_i)^2 \right)\\
        &\quad \times \left( \int_0^{1-\sum_{i=1}^{k-r} t_i} g(kt_{k-r+1}) \cdots \int_0^{1-\sum_{i=1}^{k-1} t_i} g(kt_k) \d t_k \cdots \d t_{k-r+1}  \right)^2 \d t_1 \cdots \d t_{k-r}\\
    &\leq \left( \frac{\log k}{k} \right)^r \idotsint \limits_{\sum_{i=1}^{k-r} t_i \leq 1} \left( \prod_{i=1}^{k-r} g(kt_i)^2 \right)\\
    &\quad \times \left( \int_0^{1-\sum_{i=1}^{k-r} t_i} g(kt_{k-r+1})^2 \cdots \int_0^{1-\sum_{i=1}^{k-1} t_i} g(kt_k)^2 \d t_k \cdots \d t_{k-r+1}  \right) \d t_1 \cdots \d t_{k-r}\\
    &= \left( \frac{\log k}{k} \right)^r I_k(F_k).
    \end{align*}
    By the Stone-Weierstrass Theorem, we take $F(t_1,\ldots,t_k)$ to be a smooth approximation to $F_k(\rho \delta t_1,\ldots,\rho \delta t_k)$ such that
    \begin{align*}
        I_k(F)&=(\delta \rho)^k (1+O((\log k)^{-1/2})) I_k(F_k)\\
        J_k^{(r)}(F)&=(\delta \rho)^{k+r} (1+O((\log k)^{-1/2})) J_k^{(r)}(F_k)
    \end{align*}
    for all $r \in \mathbb{Z}^+$, and we are done.
\end{proof}
\begin{lem}\label{lem:maynardtaosieve}
    Let $q \ge 2$ be squarefree and $N$ sufficiently large in terms of $q$. Suppose $\{h_1,\ldots,h_k\} \subseteq [0,N]$ is an admissible $k$-tuple such that for all $1 \leq i<j \leq k$, we have $\gcd(h_i,q)=1$ and
    \begin{align*}
        p \mid h_i-h_j &\implies p \leq \e \log N.
    \end{align*}
    Let $b \in \mathbb{Z}$ such that for all $j \in \{1,\ldots,k\}$, we have $\gcd(qb+h_j,W)=1$. Then the following are true.
    \begin{enumerate}
        \item We have
        $$\sum_{\substack{N<n \leq 2N\\ n \equiv b \Mod{W}}} w_n = (1+o(1)) \frac{N}{W} B^{-k} I_k(F).$$
        \item For each $j \in \{1,\ldots,k\}$, we have
        $$\sum_{\substack{N<n \leq 2N\\ n \equiv b \Mod{W}}} \bm{1}_{\mathbb{P}}(qn+h_j) w_n = (1+o(1)) \frac{N}{W} B^{-k} J_k^{(1)}(F).$$
        \item For $r \in \{1,2,\ldots,k\}$ and $j_1,\ldots,j_r \in \{1,\ldots,k\}$ strictly increasing, we have
        $$\sum_{\substack{N<n \leq 2N\\ n \equiv b \Mod{W}}} \bm{1}_{\mathbb{P}} (qn+h_{j_1}) \cdots \bm{1}_{\mathbb{P}}(qn+h_{j_r}) w_n \leq (4^{r-1} (r-1)^{r-1}+O(\delta)) \frac{N}{W} B^{-k} J_k^{(r)}(F).$$
    \end{enumerate}
\end{lem}
\begin{proof}
    Parts (1) and (2) is nearly identical to \cite[Lemma 4.6]{banks2016limit}, and any differences can be found in the proof of (3), so we only prove (3). We will use the sieve upper bound
    $$\bm{1}_{\mathbb{P}} (qn+h_{j_i}) \leq \left( \sum_{e_i \mid qn+h_{j_i}} \mu(e_i) G_i \left( \frac{\log e_i}{\log N} \right) \right)^2$$
    for smooth decreasing functions $G_i:[0,\infty) \to \mathbb{R}$ supported on $[0,\frac{1}{4(r-1)}-\frac{2 \delta}{r-1}]$ with $G(0)=1$, for each $i=1,2,\ldots,r-1$. Observe that there is no contribution from any summand in the definition of $w_n$
unless $d_{j_1}=\cdots=d_{j_r}=1$, since
$1_{\mathbb P}(qn+h_{j_i})$ together with the divisibility condition
$d_{j_i}\mid(qn+h_{j_i})$ forces $d_{j_i}=1$ for each
$i=1,\ldots,r$. Thus, we have
    \small
    \begin{align*}
        &\sum_{\substack{N<n \leq 2N\\ n \equiv b \Mod{W}}} \prod_{i=1}^r \bm{1}_{\mathbb{P}}(qn+h_{j_i}) \left( \sum_{\substack{d_1,\ldots, d_k\\ d_i \mid qn+h_{j_i} \forall i}} \lambda_{d_1 \cdots d_k} \right)^2\\
        &\leq \sum_{\substack{N<n \leq 2N\\ n \equiv b \Mod{W}}} \bm{1}_{\mathbb{P}}(qn+h_{j_r}) \prod_{i=1}^{r-1} \left( \sum_{e_i \mid qn+h_{j_i}} \mu(e_i) G_i \left( \frac{\log e_i}{\log N} \right) \right)^2 \Bigg( \sum_{\substack{d_1,\ldots, d_k\\ d_i \mid qn+h_{j_i} \forall i\\ d_{j_1} = \cdots = d_{j_r}=1}} \lambda_{d_1 \cdots d_k} \Bigg)^2\\
        &= \sum_{\substack{N<n \leq 2N\\ n \equiv b \Mod{W}}} \bm{1}_{\mathbb{P}}(qn+h_{j_r}) \Bigg( \sum_{\substack{d_1,\ldots,d_k\\e_1,\ldots,e_{r-1}\\ d_i \mid qn+h_{j_i} \forall i\\ d_{j_1} = \cdots = d_{j_r}=1\\ e_\ell \mid qn+h_{j_\ell} \forall \ell}} \lambda_{d_1,\ldots,d_k} \mu(e_1) \cdots \mu(e_{r-1}) G_1 \left( \frac{\log e_1}{\log N} \right) \cdots G_{r-1} \left( \frac{\log e_{r-1}}{\log N} \right) \Bigg)^2.
    \end{align*}
    \normalsize
    Expanding the square,
    \small
    \begin{align*}
        &= \sum_{\substack{N<n \leq 2N\\ n \equiv b \Mod{W}}} \bm{1}_{\mathbb{P}}(qn+h_{j_r}) \sum_{\substack{d_1,\ldots,d_k\\e_1,\ldots,e_{r-1}\\ d_i \mid qn+h_{j_i} \forall i\\ d_{j_1}=\cdots=d_{j_r}=1\\ e_\ell \mid qn+h_{j_\ell} \forall \ell}} \sum_{\substack{d_1',\ldots,d_k'\\e_1',\ldots,e_{r-1}'\\ d_i' \mid qn+h_{j_i} \forall i\\ d_{j_1}'=\cdots=d_{j_r}'=1\\ e_\ell' \mid qn+h_{j_\ell} \forall \ell}} \lambda_{d_1,\ldots,d_k} \lambda_{d_1',\ldots,d_k'}\\
        &\hspace{8.5cm} \times \prod_{i=1}^{r-1} \mu(e_i)\mu(e_i')G_i \left( \frac{\log e_i}{\log N} \right) G_i \left( \frac{\log e_i'}{\log N} \right)\\
        &= \sum_{\substack{d_1,\ldots,d_k\\e_1,\ldots,e_{r-1}\\d_{j_1}=\cdots=d_{j_r}=1\\\gcd(d_i,q)=1 \forall i\\ \gcd(e_\ell,q)=1 \forall \ell}} \sum_{\substack{d_1',\ldots,d_k'\\e_1',\ldots,e_{r-1}'\\ d_{j_1}'=\cdots=d_{j_r}'=1\\\gcd(d_i',q)=1 \forall i\\ \gcd(e_\ell',q)=1 \forall \ell}} \lambda_{d_1,\ldots,d_k} \lambda_{d_1',\ldots,d_k'} \prod_{i=1}^{r-1} \mu(e_i)\mu(e_i')G_i \left( \frac{\log e_i}{\log N} \right) G_i \left( \frac{\log e_i'}{\log N} \right)\\
        &\hspace{4.5cm} \times \sum_{\substack{N<n \leq 2N\\ n \equiv b \mod{W}\\ n \equiv -q^{-1} h_{j_i} \mod{[d_i,d_i']} \forall i\\ n \equiv -q^{-1} h_{j_\ell} \mod{[e_\ell,e_\ell']} \forall \ell}} \bm{1}_{\mathbb{P}}(qn+h_{j_r}),
    \end{align*}
    \normalsize
    since we assumed $\gcd(q,h_{j_i})=1$ for all $ 1\leq i \leq r-1$. The innermost sum is
    $$\frac{\pi(2qN+h_{j_r})-\pi(qN+h_{j_r})}{\varphi(qW) \prod_{i=1}^k \varphi([d_i,d_i']) \prod_{i=1}^{r-1} \varphi([e_i,e_i'])}+O \left(E \left(qN;qW\prod_{i=1}^k [d_i,d_i'] \prod_{i=1}^{r-1} [e_i,e_i']\right) \right),$$
    where
    $$E(qN;q')=\max_{\substack{(a,q')=1\\ h \in \mathcal{H}}} \left| \pi(2qN+h;q',a)-\pi(qN+h;q',a)-\frac{\pi(2qN+h)-\pi(qN+h)}{\varphi(q')} \right|,$$
    because by the support of $\lambda_{d_1,\ldots,d_k}$ and the choice of $b$ we have $[d_i,d_i'],[e_\ell,e_\ell']$ are all pairwise coprime, and by assumption $q \mid W$. We first deal with the error term. In the same way as \cite[Lemma 4.6(iii)]{banks2016limit}, we can restrict to arithmetic progressions mod $sW$, where $$s=\prod_{i=1}^k [d_i,d_i'] \prod_{i=1}^{r-1} [e_i,e_i'] \leq N^{1/2-\delta}.$$
    Using the bound $\lambda_{d_1,\ldots,d_k} \ll 1$ and the trivial bound $E(qN;q') \ll 1+qN/\varphi(q')$, using Cauchy--Schwarz and Theorem \ref{modifiedbombierivinogradov}, the error term contributes
    \begin{align*}
        &\sideset{}{'}\sum_{\substack{d_1,\ldots,d_k\\d_1',\ldots,d_k'\\d_{j_1}=\cdots=d_{j_r}=1\\d_{j_1}' = \cdots = d_{j_r}'=1\\e_1,\ldots,e_{r-1}\\ e_1',\ldots,e_{r-1}'}}  |\lambda_{d_1,\ldots,d_k}\lambda_{d_1',\ldots,d_k'}| E \left(qN;qW\prod_{i=1}^k [d_i,d_i'] \prod_{i=1}^{r-1} [e_i,e_i']\right)\\
        &\ll \sum_{\substack{s \leq N^{\frac12- \delta}\\ \gcd(s,WZ_{N^{4\e}})=1}} \mu(s)^2 \tau_{6k}(s) E \left(qN;sqW\right)\\
        &\ll \Bigg( \sum_{\substack{s \leq N^{\frac12- \delta}\\ \gcd(s,WZ_{N^{4 \e}})=1}} \mu(s)^2 \tau_{6k}(s)^2 \left( 1+\frac{qN}{\varphi(sqW)} \right) \Bigg)^{1/2} \Bigg( \sum_{\substack{s \leq N^{\frac12- \delta}\\ \gcd(s,WZ_{N^{4 \e}})=1}} \mu(s)^2 E \left(qN;sqW\right) \Bigg)^{1/2}\\
        &\ll  \frac{N}{W(\log N)^{2k}},
    \end{align*}
    where $\sum'$ denotes the additional pairwise coprimality condition between $[d_i,d_i'],[e_\ell,e_\ell'],qWZ_{N^{2\e}}$ and $\tau_{6k}(s)$ denotes the number of ordered $6k$-tuples of positive integers whose product is $s$. The main term is treated the same as \cite[Lemma 4.6(ii)]{banks2016limit}. Expanding $\lambda_{d_1,\ldots,d_k}$, the main term is
    \begin{align*}
        (1+o(1))\frac{qN}{\log N} \sum_{j=1}^J \sideset{}{'}\sum_{\substack{d_1,\ldots,d_k\\ d_1',\ldots,d_k'\\ d_{j_1} = \cdots = d_{j_r}=1\\d_{j_1}' = \cdots = d_{j_r}'=1}} &\prod_{j=1}^k \mu(d_j) \prod_{\substack{1 \leq \ell \leq k\\ \ell \neq j_1,\ldots,j_{r-1}}} F_{\ell,j}\left( \frac{\log d_\ell}{\log N} \right) \prod_{i=1}^{r-1} F_{j_i,j}(0)\, G_i \left( \frac{\log d_i}{\log N} \right)\\
        &\prod_{j=1}^k \mu(d_j') \prod_{\substack{1 \leq \ell \leq k\\ \ell \neq j_1,\ldots,j_{r-1}}} F_{\ell,j}\left( \frac{\log d_\ell'}{\log N} \right) \prod_{i=1}^{r-1} F_{j_i,j}(0)\, G_i \left( \frac{\log d_i'}{\log N} \right)\\
        &\times \varphi(qW)^{-1} \varphi([d_1,d_1'])^{-1} \cdots \varphi([d_k,d_k'])^{-1},
    \end{align*}
    where $\sum'$ denotes the additional restriction of $[d_1,d_1'],\ldots,[d_k,d_k'],qWZ_{N^{4\e}}$ being pairwise coprime. Let
    \begin{align*}
        \widetilde{F}(t_1,\ldots,t_k)=&G_1'(t_{j_1}) \cdots G_{r-1}'(t_{j_{r-1}}) \\
        &\times \int_0^\infty \cdots \int_0^\infty F(t_1,\ldots,t_{j_1-1},u_{j_1},\ldots,u_{j_{r-1}},t_{j_r+1},\ldots,t_k) \d u_{j_1} \cdots \d u_{j_{r-1}}.
    \end{align*}
    Note $\widetilde{F}$ is supported on $t_1,\ldots,t_k$ with $\sum_{i=1}^k t_i \leq 1/4-\delta$. Using Lemma \ref{mobiusinvlem}, the main term is
    \begin{align*}
        &(1+o(1)) \frac{qN}{\varphi(qW) \log N} B^{-k+1} \sum_{j=1}^J \prod_{\substack{1 \leq \ell \leq k\\ \ell \neq j_1,\ldots,j_r}} \int_0^\infty F'_{\ell,j}(t_\ell)^2 \d t_\ell \prod_{i=1}^r F_{j_i,j}(0)^2 \prod_{i=1}^{r-1} \int_0^\infty G_i'(t_i)^2 \d t_i\\
        &\leq (1+o(1)) \frac{N}{W} B^{-k} \int_0^\infty \cdots \int_0^\infty \left( \int_0^\infty \widetilde{F}(t_1,\ldots,t_k) \d t_{j_r} \right)^2 \d t_1 \cdots \d t_{j_r-1} \d t_{j_r+1} \cdots \d t_k.
    \end{align*}
    Combined with the above error term bound we have
    \begin{align*}
        &\sum_{\substack{N<n \leq 2N\\ n \equiv b \Mod{W}}} \prod_{i=1}^r \bm{1}_{\mathbb{P}}(qn+h_{j_i}) \Bigg( \sum_{\substack{d_1,\ldots,d_k\\ d_i \mid qn+h_{j_i} \forall i\\ d_{j_1}=\cdots=d_{j_r}=1}} \lambda_{d_1,\ldots,d_k} \Bigg)^2\\
        &\leq (1+o(1)) \frac{N}{W} B^{-k} \int_0^\infty \cdots \int_0^\infty \left( \int_0^\infty \widetilde{F}(t_1,\ldots,t_k) \d t_{j_r} \right)^2 \d t_1 \cdots \d t_{j_r-1} \d t_{j_r+1} \cdots \d t_k\\
        &=(1+o(1)) \frac{N}{W}B^{-k} J_k^{(r)}(F) \int_0^\infty \cdots \int_0^\infty \prod_{i=1}^{r-1} G_i'(t_{j_i})^2 \d t_{j_1} \cdots \d t_{j_{r-1}}.
    \end{align*}
    Taking $G_i(t)$ to be a fixed smooth approximation to $1-t/(\frac{1}{4(r-1)}-\frac{2 \delta}{r-1})$ with $G(0)=1$ and $\int_0^\infty G_i'(t)^2 \d t \leq 4(r-1)+O(\delta)$, we are done.
\end{proof}
It will be helpful to make the following definitions.
\setcounter{section}{1}
\setcounter{thm}{3}
\begin{defn}
    For a squarefree{} integer $q \ge 2$, $m \in \mathbb{Z}^+$, and $A \subseteq \prod_{i=1}^m (\mathbb{Z}/q\mathbb{Z})^\times$, define
    \[
    \pi(x;q,A)=\#\{p_n \leq x:(p_n \mod{q},\ldots,p_{n+m-1} \mod{q}) \in A\}.
    \]
    Also, for $r \in \mathbb{Z}^+$ with $r>1$, define 
    \[
    M{}=\left\lceil \left(\frac{2^{3r-2} (r-1)^{2r-1}}{r!}\right)^{\frac{1}{r-1}} m(m(r-1)+r)^{\frac{1}{r-1}}  \right\rceil,
    \]
    and define the set of functions
    \small
    $$J_r(m,M)=\{j:\{1,\ldots,m\} \to \{1,\ldots,M\}:j(i+1) \geq j(i), \text{ no consecutive $r$ values }j(i) \text{ are equal}\}.$$
    \normalsize
\end{defn}
\setcounter{section}{5}
\setcounter{thm}{5}
\begin{prop} \label{maynardtaosieve}
    Let $m,r \in \mathbb{Z}^+$ with $r>1$, $q \ge 2$ be a squarefree{} integer and recall $M$ from Definition \ref{defn:notationforthms}.
    Let $k$ be a sufficiently large multiple of $M{}$ in terms of $m$ and $r$. Let $\e>0$ be sufficiently small. Define $$W := \prod_{\substack{p\leq \e \log N\\ p \nmid Z_{N^{4 \e}}}}p.$$ Let $\mathcal{H}= \{h_1,\ldots,h_k\} \subseteq [0,N]$ be an admissible $k$-tuple such that for all $1 \leq i<j \leq k$, $\gcd(h_i,q)=1$ and $$p \mid h_i-h_j \implies p \leq \e \log N.$$ Let $b \in \mathbb{Z}$ such that
    $$\left( \prod_{j=1}^k (qb+h_j),W \right)=1.$$
    Let $\mathcal{H}= \mathcal{H}_1 \cup \cdots \cup \mathcal{H}_{M{}}$ be a partition of $\mathcal{H}$ into $M{}$ sets of equal size. Then for all sufficiently large $N$ in terms of $m,q,k,\e$, there is $n \in [N,2N]$ with $n \equiv b \Mod{W}$ and some set of distinct indices $\{i_1,\ldots,i_{m+1}\}$ such that
    \begin{align*}
        1 \leq |\mathcal{H}_{i}(n) \cap \mathbb{P}| \leq r-1 &\quad \text{ for all } i \in \{i_1,\ldots,i_{m+1}\},\\
        |\mathcal{H}_i(n) \cap \mathbb{P}| =0 &\quad \text{ for all } i_1<i<i_{m+1} \text{ such that } i \neq i_1,\ldots,i_{m+1}.
    \end{align*}
\end{prop}
\begin{proof}
     The proof is similar to that of \cite[Theorem 4.3]{banks2016limit}. Consider
     \small
     $$S=\sum_{N<n \leq 2N} \Bigg( \sum_{i=1}^k \bm{1}_{\mathbb{P}}(qn+h_i)-m(r-1)-(m(r-1)+r) \sum_{j=1}^M{} \sum_{\substack{h_{j_1},\ldots,h_{j_r} \in \mathcal{H}_j\\ j_1<\cdots<j_r}} \prod_{i=1}^r \bm{1}_{\mathbb{P}} (qn+h_{j_i}) \Bigg) w_n.$$
     \normalsize
     By Lemma \ref{lem:maynardtaosieve}, we have
     \begin{align*}
         S \geq \sum_{i=1}^k &(1+o(1)) \frac{N}{W} B^{-k} J_k^{(1)}(F)-m(r-1)(1+o(1)) \frac{N}{W} B^{-k} I_k(F)\\
         &-(m(r-1)+r) \sum_{\substack{h_{j_1},\ldots,h_{j_r} \in \mathcal{H}_j\\ j_1<\cdots<j_r}} (4^{r-1}(r-1)^{r-1}+o(\delta)) \frac{N}{W} B^{-k} J_k^{(r)}(F).
     \end{align*}
     Using Lemma \ref{integralestimate} and choosing $0<\rho<1$ such that $\rho\delta \log k=2(r-1)m$, we get
     \begin{align*}
         S &\geq \frac{N}{W} B^{-k} I_k(F) \Bigg(k \cdot \frac{2(r-1)m}{k}-m(r-1)\\
         &\hspace{3.2cm} -4^{r-1}(r-1)^{r-1}(m(r-1)+r)  M{} \binom{k/M{}}{r} \left( \frac{2(r-1)m}{k} \right)^r-O(\delta) \Bigg)\\
         &> \frac{N}{W} B^{-k} I_k(F) \left( m(r-1)-4^{r-1}(r-1)^{r-1}(m(r-1)+r)  \cdot \frac{2^r(r-1)^rm^r}{r! M{}^{r-1}} \right),
     \end{align*}
     so $S>0$ since $$M{}=\left\lceil \left(\frac{2^{3r-2} (r-1)^{2r-1}}{r!}\right)^{\frac{1}{r-1}} m(m(r-1)+r)^{\frac{1}{r-1}}  \right\rceil.$$
     Therefore, there must exist $n \in (N,2N]$ making a positive contribution to $S$. Observe that
     \begin{enumerate}
         \item every $\mathcal{H}_i$ with $|\mathcal{H}_i(n) \cap \mathbb{P}| \geq r$ contributes at most $r-m(r-1)-r=-m(r-1)$ to $S$,
         \item every $\mathcal{H}_i$ with $|\mathcal{H}_i(n) \cap \mathbb{P}| \in [1,r-1]$ contributes at most $r-1$.
     \end{enumerate}
     For each $n \in (N,2N]$, define
     \begin{align*}
         s_n &= \text{ number of indices $i$ for which $|\mathcal{H}_i(n) \cap \mathbb{P}| \geq r$},\\
         t_n &= \text{ number of indices $i$ for which $|\mathcal{H}_i(n) \cap \mathbb{P}| \in [1,r-1]$}.
     \end{align*}
     For those $n$ making a positive contribution to $S$, we must have
     $$-m(r-1)s_n-m(r-1)+t_n(r-1)>0,$$
     which implies $t_n \geq m+1+ms_n$, i.e. number of indices $j$ for which $|\mathcal{H}_j(n)\cap\mathbb{P}| \in [1,r-1]$ is at least $m+1+ms_n$. In particular, there must be some set of $m+1$ indices $i_1<\cdots<i_{m+1}$ for which $|\mathcal{H}_i(n) \cap \mathbb{P}| \in [1,r-1]$ for $i=i_1,\ldots,i_{m+1}$ and $|\mathcal{H}_i(n) \cap \mathbb{P}|=0$ for $i_1<i<i_{m+1}$ and $i\neq i_1,\ldots,i_{m+1}$.
\end{proof}
\section{A Modified Erd\H{o}s--Rankin Type Construction} \label{sec:erdosrankin}
We have the following elementary lemma.
\begin{lem} \label{erdoselementarylem}
    Let $\{h_1,\ldots,h_k\}$ be an admissible $k$-tuple, let $S \subseteq \mathbb{Z}$, and $\mathcal{P}$ be a set of primes such that for some $x \geq 2$, we have
    $$\left\{ \begin{aligned}
        &\{h_1,\ldots,h_k\} \subseteq S \subseteq [0,x^2],\\
        &|\{p \in \mathcal{P}:p>x\}|>|S|+k.
    \end{aligned} \right.$$
    Then, there is a set of integers $\{a_p:p \in \mathcal{P}\}$ such that
    $$\{h_1,\ldots,h_k\}=S \setminus \bigcup_{p \in \mathcal{P}}\{g:g \equiv a_p \Mod{p} \}.$$
\end{lem}
\begin{proof}
    This is \cite[Lemma 5.1]{banks2016limit}.
\end{proof}
As in \citet{banks2016limit}, we need Mertens' 3rd Theorem: for $x \geq 2$,
\begin{align}
    \prod_{p \leq x} \left( 1-\frac{1}{p} \right)&=\frac{e^{-\gamma}}{\log x} \left( 1+O \left( \frac{1}{\log x} \right) \right), \label{merten3}
\end{align}
where $\gamma=0.5772 \ldots$ is the Euler-Mascheroni constant. Also, from \cite[Chapter 20 (13)]{davenport2013multiplicative}, for any positive constant $c$, there is a positive constant $c'$ depending only on $c$ such that
\begin{align}
    \sum_{\substack{x<p \leq x+y\\ p \equiv a \Mod{q}}} \log p=\frac{y}{\varphi(q)}+O\left(x \exp \left(-c'\sqrt{\log x} \right) \right) \label{paigeapprime}
\end{align}
uniformly for $2 \leq y \leq x,q \leq \exp(c\sqrt{\log x})$ and $\gcd(q,a)=1$, except possibly when $q$ is a multiple of some $q_1$ depending on $x$ which satisfies $P^+(q_1) \gg_c \log_2 x$.
\begin{lem} \label{lemwithmod}
    Fix $k \in \mathbb{N}$, squarefree{} integer $q \ge 2$ and integers $0 < r_1 \leq \cdots \leq r_k$ all coprime to $q$. There is a number $y'=y'(q,\mathbf{r},k)$ depending only on $q,r_1,\ldots,r_k$ and $k$ such that the following holds. Let $x,y,z \in \mathbb{R}$ satisfy $x \geq 1, y\geq y'$ and
    $$2y(1+(1+r_k)x) \leq \frac{2q}{\varphi(q)}z \leq y (\log_2 y)(\log_3 y)^{-1}.$$
    Let $\mathcal{Z}$ be any (possibly empty) set of primes such that for all $p' \in \mathcal{Z}$,
\begin{align}
    \sum_{\substack{p \in \mathcal{Z}\\ p \geq p'}} \frac{1}{p} \ll \frac{1}{p'} \ll \frac{1}{\log z}. \label{reciprocalprimecondition}
\end{align}
There is a set $\{\widetilde{a}_p:p \leq y, p \notin \mathcal{Z}\}$ and an admissible $k$-tuple $\{h_1,\ldots,h_k\} \subseteq (y,z]$ such that
\begin{align*}
    \{h_1,\ldots,h_k\}&=\left((0,z] \cap \mathbb{Z} \right) \setminus \bigcup_{\substack{p \leq y\\ p \notin \mathcal{Z}}} \{g: g \equiv \widetilde{a}_p \Mod{p}\}
\end{align*}
and $p \mid \widetilde{a}_p$ whenever $p \mid q$.
Moreover, for $1 \leq i<j \leq k$,
$$p \mid h_i-h_j \implies p\leq y,$$
and for $1 \leq i\leq k$, $h_i \equiv r_i \Mod{q}$.
\end{lem}
\begin{proof}
    The proof is very similar to the proof of \cite[Lemma 5.2]{banks2016limit}. Let $y_1,y_2,y,z$ be numbers satisfying
    $$2<y_1<y_2<y<z<y_1y_2, \quad 2 \log y_1 \leq (\log z)(\log_2 z)^{-1}.$$
    Let $\mathcal{Z}$ be any set consisting of primes satisfying (\ref{reciprocalprimecondition}). As in \cite{banks2016limit}, we have the following estimates on $\mathcal{Z}$:
    \begin{align}
        \prod_{p \in \mathcal{Z}} \left( 1-\frac{1}{p} \right)^{-1} &= 1+O\left( \frac{1}{\log z} \right), \label{merten3z}\\
        \sum_{\substack{p \in \mathcal{Z}\\ p \leq y_0}} 1 &\ll \log y_0 \text{ for all }y_0 \ge 1. \label{pntz}
    \end{align}
    For $y$ (and hence $\log z$) sufficiently large, from \eqref{reciprocalprimecondition} we may assume $2 \notin \mathcal{Z}$. Suppose further that $y_1>P^+(q)$. Let
    $$P_1 = \prod_{\substack{2<p \leq y_1\\ p \notin \mathcal{Z},p \neq \ell}} p, \quad P_2=\prod_{\substack{y_1<p \leq y_2\\ p \notin \mathcal{Z}}} p, \quad P_3=\prod_{\substack{y_2<p \leq y\\p \notin \mathcal{Z}}} p,$$
    where $\ell$ is a prime satisfying $\ell \gg \log y_1$ chosen later. For $p \mid P_2$, we choose $\widetilde{a}_p=0$, and let
    $$\mathcal{N}_1=((0,z] \cap \mathbb{Z}) \setminus \bigcup_{p \mid P_2} \{g:g \equiv \widetilde{a}_p \Mod{p} \}=\{h \in (0,z]:\gcd(h,P_2)=1 \}.$$
    From the proof of \cite[Lemma 5.2]{banks2016limit}, we get
    $$|\mathcal{N}_1| \leq \frac{z}{\log y_2} (\log(z/y_2)+O(1)).$$
     For $p \mid P_1$ and $p \nmid q$, we choose $\widetilde{a}_p$ greedily as in \citet{banks2016limit}, which is, for any finite set $S \subseteq \mathbb{Z}$,
     $$|S|=\sum_{a \Mod{p}} \sum_{\substack{g \in S\\ g \equiv a \Mod{p}}} 1,$$
     so there is an integer $\widetilde{a}_p$ such that
     $$|\{g \in S:g \equiv \widetilde{a}_p \Mod{p} \}| \geq \frac{|S|}{p},$$
     For $p \mid q$, set $\widetilde{a}_p=0$. Repeating this process with $p$ varied over all prime divisors of $P_1$, we obtain the set
    \begin{align*}
        \mathcal{N}_2 &=\mathcal{N}_1 \setminus \bigcup_{p \mid P_1} \{g:g \equiv \widetilde{a}_p \Mod{p}\}\\
        &=\mathcal{N}_1 \setminus 
 \Bigg[ \bigcup_{\substack{p \mid P_1\\ p \nmid q}} \{g:g \equiv \widetilde{a}_p \Mod{p} \} \bigcup_{\substack{p \mid P_1\\ p \mid q}} \{g:g \equiv 0 \Mod{p}\}\Bigg]
    \end{align*}
whose cardinality satisfies the bound
$$|\mathcal{N}_2| \leq |\mathcal{N}_1| \prod_{\substack{p \mid P_1\\ p \nmid q}} \left( 1-\frac{1}{p} \right) \leq 2e^{-\gamma} \frac{qz (\log(z/y_2)+O(1))}{\varphi(q) (\log y_1)(\log y_2)}$$
by Mertens' theorem (\ref{merten3}) and (\ref{merten3z}). By the prime number theorem,
\begin{align*}
    \pi(y)-\pi(y_2) &=\frac{y}{\log y}+O \left( \frac{y}{(\log y)^2}+\frac{y_2}{\log y_2} \right)\\
    &\geq \frac{y}{\log y_2}-O \left( \frac{y_2}{\log y_2}\right)
\end{align*}
if $y_2 \ge y\log(y/y_2)/\log y$, which is implied by $y_2 \ge y\log_2y/\log y$.
By (\ref{pntz}), we have
\begin{align*}
    |\{p \in (y_2,y]:p \notin \mathcal{Z}\}|-|\mathcal{N}_2| &\geq \frac{y}{\log y_2} \left( 1-2e^{-\gamma} \frac{qz \log(z/y_2)}{\varphi(q)y \log y_1} \right)\\
    &\qquad -O \left( \frac{y_2}{\log y_2}+\frac{z}{(\log y_1)(\log y_2)} \right).
\end{align*}
We now assume
$$y_1=(\log y)^{1/4}, \quad y_2=y(\log_3 y)^{-1}, \quad y<\frac{2qz}{\varphi(q)} \leq y(\log_2 y)(\log_3 y)^{-1}.$$
Substituting, we have
$$|\{p \in (y_2,y]:p \notin \mathcal{Z}\}|-|\mathcal{N}_2| \geq \frac{y}{\log y}(1-e^{-\gamma})-O \left( \frac{y}{(\log y)(\log_3 y)} \right),$$
which tends to infinity as $y \to \infty$, so
$$|\{p \in (y_2,y]:p \notin \mathcal{Z}\}|>|\mathcal{N}_2|+k$$
for $y$ sufficiently large in terms of $k$, which we assume. Applying Lemma \ref{erdoselementarylem}, if $\{h_1,\ldots,h_k\}$ is any admissible $k$-tuple contained in $\mathcal{N}_2$, then there exist integers $\{\widetilde{a}_p:p \mid 2 \ell P_3\}$ such that
$$\{h_1,\ldots,h_k\}=\mathcal{N}_2 \setminus \bigcup_{p \mid 2 \ell P_3}\{g:g \equiv \widetilde{a}_p \Mod{p} \}.$$
Note $\{p \leq y:p \notin \mathcal{Z} \}=\{p \leq y:p \mid 2 \ell P_1P_2P_3\}$, so we are done if we can show there exists an admissible $k$-tuple $\{h_1,\ldots,h_k\} \subseteq \mathcal{N}_2$ satisfying the required conditions. To do this, we first define an arithmetic progression mod $[q,P_1]$. For each prime $p \mid [q,P_1]$, let $A_i \Mod{p}$ be defined by
$$A_i=\begin{cases}
    -1, &\quad \text{ if } \widetilde{a}_p \equiv 1 \Mod{p}, p \nmid q,p \mid P_1,\\
    1, &\quad \text{ if } \widetilde{a}_p \equiv -1 \Mod{p}, p\nmid q,p \mid P_1,\\
    r_i, &\quad \text{ if } p \mid q, 
\end{cases}$$
Since these congruence conditions are prescribed modulo the distinct prime divisors of $[q,P_1]$, the Chinese Remainder Theorem yields a unique residue class $A_i \Mod{[q,P_1]}$
satisfying all of them simultaneously.
Suppose we could choose $h_i$ to be distinct primes in $(y,z]$ congruent to $A_i \Mod{[q,P_1]}$. Then, $h_i \in \mathcal{N}_1$ implies $h_i \in \mathcal{N}_2$ since $\gcd(A_i,P_1)=1$. By the prime number theorem, note $P_1=e^{(1+o(1))y_1}$ as $y$ tends to infinity, so for $i \neq j$ we have
$$p \mid h_i-h_j \implies p \mid{\prod_{\substack{p \nmid q\\p \mid P_1}}p} \text{ or } p \mid \frac{h_i-h_j}{\prod_{\substack{p \nmid q\\p \mid P_1}}p} \implies p \leq \max\{ y_1,qz/P_1 \}<y,$$
if $y$ is sufficiently large. Also, $\{h_1,\ldots,h_k\}$ is admissible since $\min\{h_1,\ldots,h_k\} \geq y>k$, which we assume. Therefore, we are left to show that we can find $k$ distinct primes in $(y,z]$ each congruent to $A_i \pmod{[q,P_1]}$.\\

To show this, note Chebyshev's bound implies $\sum_{p \leq y_1} \log p \ll 2y_1$, so $[q,P_1]<e^{3(\log y)^{1/4}}$. Therefore, by (\ref{paigeapprime}), for each $1 \leq i \leq k$ we have
$$\sum_{\substack{u \leq p \leq u+\Delta\\ p \equiv A_i \Mod{[q,P_1]}}} \log p=\frac{\Delta}{\varphi([q,P_1])}+O\left(y \exp \left(-c'\sqrt{\log y} \right) \right)$$
uniformly for $2 \leq \Delta \leq y \leq u \leq z$ and $c'$ an absolute constant, apart from when possibly $[q,P_1]$ is a multiple of some $q_1$ depending on $u$ satisfying $P^+(q_1) \gg_c \log_2 u \gg \log y_1$. Therefore we now pick $\ell$ such that this possibility doesn't occur. Choosing $\Delta=ye^{-(\log y)^{1/4}}$, we have
$$\sum_{\substack{u \leq p \leq u+\Delta\\ p \equiv A_i \Mod{[q,P_1]}}} \log p \gg y \exp \left( -4(\log y)^{1/4} \right)$$
uniformly for $y \leq u \leq z$, so for each $i$, the left hand side is a sum of at least $k$ primes for every $A_i$ if $y$ is sufficiently large in terms of $k$. Now assume $y$ sufficiently large in terms of $r_k$ so that
$$2(1+(1+r_k)) \leq (\log_2 y)(\log_3 y)^{-1},$$
and let $x \geq 1$ be any number such that
$$2y(1+(1+r_k)x) \leq \frac{2q}{\varphi(q)} z \leq y(\log_2 y)(\log_3 y)^{-1}.$$
Let $$u=r_k xy+y,$$ so that the interval $(u,u+\Delta]$ is contained in $(y,z]$. For $1 \leq i \leq k$, we choose $h_i$ to be distinct primes in $(u,u+\Delta]$ such that $h_i \equiv A_i \Mod{[q,P_1]}$, and this is possible since each arithmetic progression $A_i \Mod{[q,P_1]}$ contains at least $k$ primes in the interval. Therefore, we are done.
\end{proof}

\section{Proofs of Main Result} \label{mainresultsection}
\setcounter{section}{1}
\setcounter{thm}{4}
In this section we combine the results of Sections \ref{sec:modifiedmaynardtaosieve} and \ref{sec:erdosrankin} to prove Theorem \ref{mainmodthm}.
\begin{thm}
    Let $q \ge 2$ be a squarefree{} integer, and $r,m \in \mathbb{Z}^+$ with $r>1$. Recall $\pi(x;q,A)$, $M$, and $J_r(m,M)$ from Definition \ref{defn:notationforthms}.
    For any $a_1,\ldots,a_{M{}} \in (\mathbb{Z}/q\mathbb{Z})^\times$, let
    $$A=\{(c_1,\ldots,c_m): \exists j \in J_r(m,M) \text{ s.t. }c_i=a_{j(i)} \forall 1 \leq i \leq m\}.$$
    Then $\pi(x;q,A) \to \infty$ as $x \to \infty$.
\end{thm}
\begin{proof}
    The case $m=1$ is known. Fix $k \geq m \geq 2$, let $\e>0$ be sufficiently small, and let $k$ be a sufficiently large multiple of $M{}$. Let $\mathbf{r}=(r_1,\ldots,r_k) \in \mathbb{R}^k$ be given by
    $$\mathbf{r}=(a_1 \mod{q},\ldots,a_1 \mod{q},a_2 \mod{q},\ldots,a_2 \mod{q},\ldots,a_{M{}} \mod{q},\ldots,a_M{} \mod{q}),$$
    where there are  $k/M{}$ consecutive copies of each $a_i \Mod{q}$ appearing in $\mathbf{r}$. We choose suitable representatives $r_i$ mod $q$ such that $r_1 \leq \cdots \leq r_k$. Let $N$ be sufficiently large in terms of $k,m,\e$, and define parameters
    $$x=\e^{-1}, \quad y= \e \log N, \quad z=\varphi(q)y(\log_2(y))(2q\log_3(y))^{-1}.$$
    If $N$ is sufficiently large in terms of $\mathbf{r}$ and $k$ as well, then with $y'(q,\mathbf{r},k)$ as in Lemma \ref{lemwithmod}, we have
    $$x>1, \quad y \geq y'(q,\mathbf{r},k), \quad 2y(1+(1+r_k)x) \leq \frac{2q}{\varphi(q)}z \leq y(\log_2 y)(\log_3 y)^{-1}.$$
    Let $Z_{N^{4 \e}}$ be defined as in Section \ref{sec:modifiedmaynardtaosieve}, $W= \prod_{p \leq \e \log N, p \nmid Z_{N^{4 \e}}} p$, and let
    $$ \mathcal{Z}= \begin{cases}
        \emptyset ,&\quad \text{ if } Z_{N^{4 \e}}=1\\
        \{Z_{N^{4 \e}}\}, &\quad \text{ if } Z_{N^{4 \e}} \neq 1.
    \end{cases}$$
    By Lemma \ref{lemwithmod}, there is a set $\{a_p:p \leq y, p \notin \mathcal{Z}\}$ and an admissible $k$-tuple $\{h_1,\ldots,h_k\} \subseteq (y,z]$ such that
\begin{align*}
    \{h_1,\ldots,h_k\}&=\left((0,z] \cap \mathbb{Z} \right) \setminus \bigcup_{\substack{p \leq y\\ p \notin \mathcal{Z}}} \{g: g \equiv \widetilde{a}_p \Mod{p}\},
\end{align*}
and $p \mid \widetilde{a}_p$ whenever $p \mid q$.
Moreover, for $1 \leq i<j \leq k$,
$$p \mid h_i-h_j \implies p\leq y,$$
and define the partition
$$\mathcal{H}=\mathcal{H}_1 \sqcup \cdots \sqcup \mathcal{H}_{M{}}$$
such that for each $j=1,2,\ldots,M{}$ we have
\begin{equation*}
        h \equiv r_{jk/M} \Mod{q}
\end{equation*}
for all $h \in \mathcal{H}_j$. Let $b \in \mathbb{Z}$ satisfying
$$b \equiv -q^{-1}\widetilde{a}_p \Mod{p}$$
if $p \leq y, p \notin \mathcal{Z}$ and $p \nmid q$, whereas if $p \mid q$ set $b \equiv 0 \Mod{p}$. Therefore, the assumptions of Proposition \ref{maynardtaosieve} hold, and there is some $n \in (N,2N]$ with $n \equiv b \Mod{W}$ and some set of distinct indices $\{i_1,\ldots,i_{m+1}\}$ such that
    \begin{align*}
        1 \leq |\mathcal{H}_{i}(n) \cap \mathbb{P}| \leq r-1 &\quad \text{ for all } i \in \{i_1,\ldots,i_{m+1}\},\\
        |\mathcal{H}_i(n) \cap \mathbb{P}| =0 &\quad \text{ for all } i_1<i<i_{m+1} \text{ such that } i \neq i_1,\ldots,i_{m+1}.
    \end{align*}
To prove they are consecutive primes, note
$$(qn,qn+z] \cap \mathbb{P} = \mathcal{H}(n) \cap \mathbb{P},$$
since if $g \in (0,z]$ and $g \notin \{h_1,\ldots,h_k\}$, then $qn+g \equiv qb+\widetilde{a}_p \equiv -\widetilde{a}_p+\widetilde{a}_p \equiv 0\Mod{p}$ for some $p \leq y$ with $p \notin \mathcal{Z}$, so the primes in $\mathcal{H}(n)$ are consecutive primes. Therefore, there must exist consecutive primes $p_n,\ldots,p_{n+m-1} \in [qN,3qN]$ such that $p_{n+i-1} \equiv a_{j'(i)} \mod{q}$ for all $1 \le i \le m$, where $j'(i+1)\geq j'(i)$ for all $1 \le i \le m$ and no consecutive $r$ of $(j'(i))_{i=1}^m$ are congruent mod $q$. Taking $N \to \infty$, we are done.
\end{proof}
\setcounter{section}{7}
\setcounter{thm}{0}

\section{Number of Attainable Residue Class Patterns} \label{maincorsection}
In this section, we use Theorem \ref{mainmodthm} and recursive combinatorial arguments to obtain lower bounds for the number of $m$-tuples occurring infinitely often among consecutive primes.
\begin{cor} \label{smallgeneralcor}
    For $m,r \in \mathbb{Z}^+$ with $2 \leq r \leq m/100$, recall $M$ from Definition \ref{defn:notationforthms}.
    For any constant $0<c<1$, if $q$ is squarefree and $\varphi(q) \geq \lceil \frac{M{}}{\lfloor c(m-1)/(r-1) \rfloor} \rceil$, there are at least
    $$4 \lfloor (1-c)m \rfloor\left(\left\lceil \frac{M{}}{\lfloor c(m-1)/(r-1) \rfloor} \right\rceil\right)^{-5} \varphi(q)(\varphi(q)-1)$$
    $m$-tuples $\mathbf{a}$ such that $\pi(x;q,\mathbf{a}) \to \infty$ as $x \to \infty$.
\end{cor}

\begin{proof}
    Using Theorem \ref{mainmodthm}, for any $a_1,\ldots,a_{M{}}$, there must exist $\mathbf{a}=(a_{j(1)},\ldots,a_{j(m)})$ with $j$ increasing and no consecutive $r$ values the same, such that
    $$\pi(x;q,\mathbf{a}) \to \infty \text{ as } x \to \infty.$$
    We call a $m$-tuple $\mathbf{a}$ with this property `good'. Define a set $S_1$ consisting of all $M{}$-tuples with entries in $(\mathbb{Z}/q\mathbb{Z})^\times$ of the form
   $$(\underbrace{a_1,\ldots,a_1}_{\lfloor c(m-1)/(r-1) \rfloor\text{ times}},\underbrace{a_2,\ldots,a_2}_{\lfloor c(m-1)/(r-1) \rfloor\text{ times}},\cdots), \quad a_i \text{ distinct}.$$
   Note $S_1$ is non-empty since $\varphi(q) \geq \lceil \frac{M{}}{\lfloor c(m-1)/(r-1) \rfloor} \rceil$ by assumption. We pick good $m$-tuples with the following recursive process.
   \begin{enumerate}
       \item Take a $M$-tuple $(a_1,\ldots,a_1,a_2,\ldots) \in S_1$. By Theorem \ref{mainmodthm}, there is a good $m$-tuple of the form
       $$(b_1,\ldots,b_1,b_2,\ldots,b_2,\ldots,b_{\ell_1},\ldots,b_{\ell_1}),$$
       where $2\leq \ell_1 \leq \lceil \frac{M{}}{\lfloor c(m-1)/(r-1) \rfloor} \rceil$.
       \item Define
       \begin{align*}
       S_2 := S_1 \setminus \{ (a_1,\ldots,a_1,a_2,\ldots,a_2,\ldots):&\exists i,i',j \text{ s.t. } i<i' \text{ and } a_i=b_j,a_{i'}=b_{j+1}\}.
       \end{align*}
       \item Take any element from $S_2$, then repeat the above process until $S_k$ is empty.
   \end{enumerate}
   The good $m$-tuples obtained from this process must be piecewise constant with at least $2$ distinct entries, and no two good tuples have the same two consecutive distinct entries in the same order. To find the minimum number of good tuples obtained, note
   $$\text{number of good tuples obtained }=\text{ number of times the process repeated,}$$
   which can be minimised if at each step $k$ a good $m$-tuple $(b_1^{(k)},\ldots,b_1^{(k)},\ldots,b_{\ell_k}^{(k)},\ldots,b_{\ell_k}^{(k)})$ is obtained such that $$S_{k} \cap\{ (a_1,\ldots,a_1,a_2,\ldots,a_2,\ldots):\exists i,i',j \text{ s.t. } i<i' \text{ and } a_i=b_j^{(k)},a_{i'}=b_{j+1}^{(k)}\}$$
   is maximised. However, the size of this set is clearly at most the size of
   \begin{align*}
       &S_1 \cap\{ (a_1,\ldots,a_1,a_2,\ldots,a_2,\ldots):\exists i,i',j \text{ s.t. } i<i' \text{ and } a_i=b_j^{(k)},a_{i'}=b_{j+1}^{(k)}\}
   \end{align*}
   Therefore, the number of elements removed every time is
   \begin{align*}
       &\leq \#(\text{choices of $j$}) \cdot \#(\text{choices of $i,i'$})\\
       &\quad \ \#(\text{choices of $a_k$ for $k \neq i,i'$}) \cdot \#(\text{choices of order for $a_i$ with $a_i$ before $a_{i'}$})\\
       &\leq (\ell_k-1)\binom{\lceil \frac{M{}}{\lfloor c(m-1)/(r-1) \rfloor} \rceil}{2} \binom{\varphi(q)-2}{\lceil \frac{M{}}{\lfloor c(m-1)/(r-1) \rfloor} \rceil-2} \cdot \left \lceil \frac{M{}}{\lfloor c(m-1)/(r-1) \rfloor} \right \rceil! \cdot \frac{1}{2}.
   \end{align*}
   To maximise the number of elements removed, we suppose for all $k$, $\ell_{k}=\lceil \frac{M{}}{\lfloor c(m-1)/(r-1) \rfloor} \rceil$, since this is the greatest possible value of $\ell_{k}$. Repeating this process until it terminates, the number of good tuples obtained in this way is
   \begin{align*}
       &\geq 4\binom{\varphi(q)}{\lceil \frac{M{}}{\lfloor c(m-1)/(r-1) \rfloor} \rceil} \binom{\varphi(q)-2}{\lceil \frac{M{}}{\lfloor c(m-1)/(r-1) \rfloor} \rceil-2}^{-1} \left(\left\lceil \frac{M{}}{\lfloor c(m-1)/(r-1) \rfloor} \right\rceil\right)^{-3}\\
       &\geq 4\left(\left\lceil \frac{M{}}{\lfloor c(m-1)/(r-1) \rfloor} \right\rceil\right)^{-5} \varphi(q)(\varphi(q)-1).
   \end{align*}
   By Dirichlet's Theorem on primes in arithmetic progressions, for each $a \in (\mathbb{Z}/q\mathbb{Z})^\times$, there are infinitely many primes $p \equiv a \Mod{q}$. Therefore, for each good $m$-tuple $$\mathbf{a}=(a_1,\ldots,a_1,\ldots,a_\ell,\ldots,a_{\ell})$$ obtained from the above process, by the pigeonhole principle we can create another good tuple by shifting: there exists $a \in (\mathbb{Z}/q\mathbb{Z})^\times$ such that
   $$\mathbf{a}' := (a_1,\ldots,a_1,\ldots,a_\ell,\ldots,a_\ell,a)$$
   is good, and we can keep shifting the resultant good tuple to get another good tuple. \\
   
   Let $G_0 \subseteq \prod_{i=1}^m (\mathbb{Z}/q\mathbb{Z})^\times$ be the set of good tuples obtained from the above recursive process, and let $G_i$ be the set of good tuples obtained from shifting each tuple in $G_0$ $i$ times. We claim that $$\left|\bigcup_{i=0}^{\lfloor (1-c)m \rfloor} G_{i}\right|=\lfloor (1-c)m \rfloor|G_0|,$$ i.e. all good tuples obtained from shifting at most $\lfloor (1-c)m \rfloor$ times are distinct. Indeed, observe from steps (1) and (2), $G_0$ has the following property: If $\mathbf{a},\mathbf{b} \in G_0$, then there does not exist $1 \leq i,j \leq m$ such that $a_i=b_j$ and $a_{i+1}=b_{j+1}$. Also, $a_m$ appears in $\mathbf{a}$ at most $\lfloor c(m-1)/(r-1) \rfloor \cdot (r-1) \leq c(m-1)<cm$ times. Therefore, if we shift $\mathbf{a}$ and $\mathbf{b}$ each at most $\lfloor (1-c)m \rfloor$ times to obtain $\mathbf{a}'$ and $\mathbf{b}'$ respectively, we must have $(a_1',a_2') \neq (b_1',b_2')$ and so $\mathbf{a}' \neq \mathbf{b}'$. Thus, shifting each $m$-tuple in $G_0$ $\lfloor (1-c)m \rfloor$ times, we obtain a total of
   $$4\lfloor (1-c)m \rfloor\left(\left\lceil \frac{M{}}{\lfloor c(m-1)/(r-1) \rfloor} \right\rceil\right)^{-5} \varphi(q)(\varphi(q)-1)$$
   $m$-tuples $\mathbf{a}$ such that $\pi(x;q,\mathbf{a}) \to \infty$ as $x \to \infty$.
\end{proof}
To simplify the final expression, we have
\setcounter{section}{1}
\setcounter{thm}{5}
\begin{cor} \label{smallcor}
     For any $0<c<1$, if $q \ge 2$ is squarefree and $\varphi(q) > 8c^{-1}e^2 (\log m)^2$, then for $m$ sufficiently large,$$\#\left\{\mathbf{a} \in \prod_{i=1}^m (\mathbb{Z}/q\mathbb{Z})^\times:\lim_{x \to \infty} \pi(x;q,\mathbf{a})=\infty \right\} \geq \frac{\lfloor (1-c)m \rfloor c^5}{256e^{10}(\log m)^{10}}\varphi(q)(\varphi(q)-1).$$
\end{cor}
\begin{proof}
    Letting $r=\lceil \log m+1 \rceil$ in Theorem \ref{mainmodthm}, we have $M{}/(m-1) < 8e^2 \log m$ as $m \to \infty$. Using Corollary \ref{smallgeneralcor}, for $m$ sufficiently large there are at least
    \begin{align*}
        &\geq \frac{4\lfloor (1-c)m \rfloor c^5}{4^5e^{10}(\log m)^{10}}\varphi(q)(\varphi(q)-1)
    \end{align*}
    $m$-tuples $\mathbf{a}$ such that $\pi(x;q,\mathbf{a}) \to \infty$ as $x \to \infty$.
\end{proof}
\begin{rem}
    \citet{shiu2000strings} showed for all $a \in (\mathbb{Z}/q\mathbb{Z})^\times$ and $\mathbf{a}=(a,\ldots,a) \in \prod_{i=1}^m (\mathbb{Z}/q\mathbb{Z})^\times$,
    $$\lim_{x \to \infty} \pi(x;q,\mathbf{a}) =\infty.$$
    From Proposition \ref{shiudirichlet}, we have
    $$\#\left\{\mathbf{a} \in \prod_{i=1}^m (\mathbb{Z}/q\mathbb{Z})^\times:\lim_{x \to \infty} \pi(x;q,\mathbf{a})=\infty\right\} \geq m\varphi(q).$$
    Therefore, Corollary \ref{smallcor} provides a better bound when $$\varphi(q)>256e^{10}c^{-5}(1-c)^{-1}(\log m)^{10}+1.$$
    To minimise this, we take $c=5/6$, and we get a better bound when
    $$\varphi(q)>3823e^{10}(\log m)^{10}+1.$$
\end{rem}
We can get a better lower bound for the number of patterns attainable by consecutive primes when $\varphi(q)$ is larger. In this case, the `shifting' argument does not generate many more good tuples, so we do not consider it here.
\begin{cor} \label{largegeneralbound}
    For $m,r \in \mathbb{Z}^+$, recall $M$ from Definition \ref{defn:notationforthms}.
    If $q \ge 2$ is squarefree and $\varphi(q) \geq M{}$, there are at least
    $$\frac{\lceil m/(r-1) \rceil!}{M{}(M{}-1)\cdots(M{}-\lceil m/(r-1) \rceil+1)} \cdot \varphi(q)(\varphi(q)-1) \cdots (\varphi(q)-\lceil m/(r-1) \rceil+1)$$
    $m$-tuples $\mathbf{a}$ such that $\pi(x;q,\mathbf{a}) \to \infty$ as $x \to \infty$.
\end{cor}
\begin{proof}
    Using Theorem \ref{mainmodthm}, for any $a_1,\ldots,a_{M{}}$, there must exist $\mathbf{a}=(a_{j(1)},\ldots,a_{j(m)})$ with $j$ increasing and no consecutive $r$ values the same, such that
    $$\pi(x;q,\mathbf{a}) \to \infty \text{ as } x \to \infty.$$
    We call a $m$-tuple $\mathbf{a}$ with this property `good'. Define a set $S_1$ consisting of all $M{}$-tuples with distinct entries in $(\mathbb{Z}/q\mathbb{Z})^\times$. We pick good $m$-tuples with the following recursive process.
   \begin{enumerate}
       \item Take an $M$-tuple $(a_1,\ldots,a_M) \in S_1$. By Theorem \ref{mainmodthm}, there is a good $m$-tuple of the form
       $$(b_1,\ldots,b_1,b_2,\ldots,b_2,\ldots,b_{\ell_1},\ldots,b_{\ell_1}),$$
       where $\lceil m/(r-1) \rceil \leq \ell_1 \leq m$.
       \item Define
       \begin{align*}
       S_2 := S_1 \setminus \{ (a_1,\ldots,a_M):&\text{ there exists increasing injection } \sigma:\{1,\ldots,\ell_1\} \to \{1,\ldots,M\}\\
       &\text{ such that } b_i=a_{\sigma(i)} \text{ for all }i\}.
       \end{align*}
       \item Take any element from $S_2$, then repeat the above process until $S_k$ is empty.
   \end{enumerate}
   The good $m$-tuples obtained from this process must be piecewise constant with at least $m/(r-1)$ distinct entries, and no two good tuples have the same two consecutive distinct entries in the same order. To find the minimum number of good tuples obtained, note
   $$\text{number of good tuples obtained }=\text{ number of times the process repeated,}$$
   which can be minimised if at each step $k$ a good $m$-tuple $(b_1^{(k)},\ldots,b_1^{(k)},\ldots,b_{\ell_{k}}^{(k)},\ldots,b_{\ell_{k}}^{(k)})$ is obtained such that
   \begin{align*}
   S_{k} \cap\{ (a_1,\ldots,a_M):&\text{ there exists increasing injection } \sigma:\{1,\ldots,\ell_k\} \to \{1,\ldots,M\}\\
       &\text{ such that } b_i^{(k)}=a_{\sigma(i)} \text{ for all }i\}
   \end{align*}
   is maximised. However, the size of this set is clearly at most the size of
   \begin{align*}
   S_{1} \cap\{ (a_1,\ldots,a_M):&\text{ there exists increasing injection } \sigma:\{1,\ldots,\ell_k\} \to \{1,\ldots,M\}\\
       &\text{ such that } b_i^{(k)}=a_{\sigma(i)} \text{ for all }i\}
   \end{align*}
   Therefore, the number of elements removed every time is
   \begin{align*}
       &\leq \#(\text{choices for }a_j \neq b_i) \cdot \#(\text{choices of order for }a_{\sigma(i)} = b_i^{(k)} \forall i).\\
       &\leq \binom{\varphi(q)-\ell_k}{M-\ell_k} \cdot \frac{M!}{\ell_k!}.
   \end{align*}
   To maximise the number of elements removed, we suppose for all $k$, $\ell_{k}=\lceil \frac{m}{r-1} \rceil$, since this is the smallest possible value of $\ell_{k}$. Repeating this process until it terminates, the number of good tuples obtained is
   \begin{align*}
       &\geq \left\lceil \frac{m{}}{r-1} \right\rceil ! \cdot \binom{\varphi(q)}{M} \binom{\varphi(q)-\lceil \frac{m}{r-1} \rceil}{M-\lceil \frac{m}{r-1} \rceil}^{-1}\\
       &\geq \binom{M}{\lceil m/(r-1) \rceil}^{-1} \varphi(q)(\varphi(q)-1)\cdots (\varphi(q)-\lceil m/(r-1) \rceil+1).
   \end{align*}
\end{proof}
Simplifying the expression, we have 
\begin{cor}
    If $q \ge 2$ is squarefree{} and $\varphi(q) > 8e^2 m\log m$, then for $m$ sufficiently large,
    $$\# \left\{ \mathbf{a} \in \prod_{i=1}^m (\mathbb{Z}/q\mathbb{Z})^\times:\lim_{x \to \infty} \pi(x;q,\mathbf{a})=\infty \right\} \gg e^{-O(m \log_2 m/\log m)} \varphi(q)^{m/\lceil \log m \rceil}.$$
\end{cor}
\begin{proof}
    Letting $r=\lceil \log m+1 \rceil$ in Theorem \ref{mainmodthm}, we have $M{}/(m-1) < 8e^2 \log m$ as $m \to \infty$. For $m$ sufficiently large, using Stirling's approximation we have
    \small
    \begin{align*}
        &\binom{M}{\lceil m/(r-1) \rceil}\\
        &= \frac{M!}{\lceil m/(r-1) \rceil! (M-\lceil m/(r-1) \rceil)!}\\
        &\ll \frac{M^M e^{-M}}{\sqrt{\lceil m/(r-1) \rceil} \lceil m/(r-1) \rceil^{\lceil m/(r-1) \rceil} e^{-\lceil m/(r-1) \rceil} (M-\lceil m/(r-1) \rceil)^{M-\lceil m/(r-1) \rceil} e^{\lceil m/(r-1) \rceil-M}}\\
        &\ll \frac{\left( \frac{M}{\lceil m/(r-1) \rceil}-1 \right)^{\lceil m/(r-1) \rceil}}{\sqrt{\lceil m/(r-1) \rceil}\left( 1-\frac{\lceil m/(r-1) \rceil}{M} \right)^{M}}\\
        &\ll \frac{(8e^2 (\log m)^2)^{m/\log m}}{\sqrt{m/\log m} \ e^{-m/\log m}} \\
        &\ll e^{O(m \log_2 m/\log m)}.
    \end{align*}
    \normalsize
    Therefore by Corollary \ref{largegeneralbound}, we are done, since for $m$ large it suffices to consider the leading order contribution. 
\end{proof}
\setcounter{section}{8}
\setcounter{thm}{1}

\vfill
\setcitestyle{numbers}
\bibliographystyle{apalike}
\bibliography{bibliography.bib}

@article{banks2016limit,
  title={On limit points of the sequence of normalized prime gaps},
  author={Banks, William D and Freiberg, Tristan and Maynard, James},
  journal={Proceedings of the London Mathematical Society},
  volume={113},
  number={4},
  pages={515--539},
  year={2016},
  publisher={Oxford University Press}
}

@article{kimmel2024consecutive,
  title={Consecutive runs of sums of two squares},
  author={Kimmel, Noam and Kuperberg, Vivian},
  journal={Journal of Number Theory},
  year={2024},
  publisher={Elsevier}
}

@article{maynard2016dense,
  title={Dense clusters of primes in subsets},
  author={Maynard, James},
  journal={Compositio Mathematica},
  volume={152},
  number={7},
  pages={1517--1554},
  year={2016},
  publisher={London Mathematical Society}
}

@article{merikoski2020limit,
  title={Limit points of normalized prime gaps},
  author={Merikoski, Jori},
  journal={Journal of the London Mathematical Society},
  volume={102},
  number={1},
  pages={99--124},
  year={2020},
  publisher={Wiley Online Library}
}

@article{lemke2016unexpected,
  title={Unexpected biases in the distribution of consecutive primes},
  author={Lemke Oliver, Robert J and Soundararajan, Kannan},
  journal={Proceedings of the National Academy of Sciences},
  volume={113},
  number={31},
  pages={E4446--E4454},
  year={2016},
  publisher={National Acad Sciences}
}

@book{davenport2013multiplicative,
  title={Multiplicative number theory},
  author={Davenport, Harold},
  volume={74},
  year={2013},
  publisher={Springer Science \& Business Media}
}

@article{shiu2000strings,
  title={Strings of congruent primes},
  author={Shiu, Daniel KL},
  journal={Journal of the London Mathematical Society},
  volume={61},
  number={2},
  pages={359--373},
  year={2000},
  publisher={Cambridge University Press}
}

@article{kimmel2025positive,
  title={Positive density for consecutive runs of sums of two squares},
  author={Kimmel, Noam and Kuperberg, Vivian},
  journal={Journal of the Institute of Mathematics of Jussieu},
  volume={24},
  number={5},
  pages={1995--2046},
  year={2025},
  publisher={Cambridge University Press}
}
\end{document}